\documentclass[11pt]{amsart}
\usepackage{amsfonts,amssymb,amsmath}
\usepackage[active]{srcltx}
\usepackage{amsmath}
\usepackage{color}

\newtheorem{theorem}{\bf Theorem}[section]

\newtheorem{remark}{\bf Remark}[section]
\newtheorem{lemma}{\bf Lemma}[section]

\numberwithin{equation}{section}

\title[Risk-sensitive Games for Continuous Time Markov Chains ]
{Zero-sum Risk-sensitive Stochastic Games for Continuous Time Markov Chains}

{\author{Mrinal K. Ghosh, K. Suresh Kumar and Chandan Pal }
\address{Department of
Mathematics, Indian Institute of Science, Bangalore -560012, India. }
\address{Department of
Mathematics, Indian Institute of Technology Bombay, Mumbai -
400076, India. }
\address{Department of
Mathematics, Indian Institute of Science, Bangalore -560012, India. }}

\email{mkg@math.iisc.ernet.in, suresh@math.iitb.ac.in, chandan14@math.iisc.ernet.in }
\begin{document}
\maketitle

\begin{abstract}
\noindent
 We study infinite horizon discounted-cost and ergodic-cost risk-sensitive zero-sum stochastic games for controlled continuous time Markov chains on a countable state space. For the discounted-cost game we prove the existence of value and saddle-point equilibrium in the class of Markov strategies under  nominal conditions. For the ergodic-cost game we prove the existence of values and saddle point equilibrium by studying the corresponding Hamilton-Jacobi-Isaacs equation under a certain Lyapunov condition. 
\end{abstract}

\vspace{10mm}

\noindent
 {\it Key words:} Risk-sensitive cost, infinite horizon discounted cost, infinite horizon ergodic cost, HJI equation, value, saddle point equilibrium.

\vspace{5mm}
\noindent
{\it 2000 Mathematics Subject Classification.}  Primary 93E20, Secondary 60J75.

\section{Introduction}
This paper is a sequel to \cite{SP1} where the risk sensitive continuous time Markov decision process
is studied on a countable state space. In this paper we extend the result of \cite{SP1} to risk-sensitive
zero sum stochastic games for continuous time control Markov chains on a countable state space. A zero-sum
risk-sensitive  differential game has been studied in \cite{BG1} and the corresponding discrete time problem
studied in \cite{BG2}. As noted in \cite{BG1} and \cite{BG2}, the zero-sum risk-sensitive stochastic dynamic
game is relevant in worst-case scenarios, for example, in financial applications when a risk-averse investor is trying to
minimize his long-term portfolio loss against the market which, by default, is antagonistic and hence the
 maximizer. As a result the minimizer chooses the risk-aversion parameter $\theta >0$
and tries to minimizes his expected risk-sensitive costs. Thus the risk-sensitive parameter is positive.
If $\theta <0$ then  minimizer would be risk-seeking. The maximizer is not risk-seeking but simply antagonistic to the minimizer.
Under certain conditions we establish value and saddle point strategies for both players.

The rest of the paper is structured as follows.  Section 2 deals with the description of the problem.
In Section 3, we prove the existence of value and saddle-point equilibrium in the class of Markov strategies
for the discounted-cost risk-sensitive zero-sum game. The analysis of ergodic-cost risk-sensitive zero-sum game
is carried out in Section 4. The paper is concluded in Section 5 with some concluding remarks.
\section{Problem Description}
Let  $U_i, i=1,2 $, be compact metric space and $ V_i=\mathcal{P}(U_i)$, space of probability measure on $U_i$ with Prohorov topology. Let
$$U:=U_1\times U_2 \; \; \mbox{and} \; \; V:=V_1\times V_2 .$$
Let $\bar{\pi}_{ij}: U \to [0,\infty) $ for $i \neq j$ and $\bar{\pi}_{ii}: U \to \mathbb{R}$ for $i \in S $. Define
$\pi_{ij} :  V \to \mathbb{R} $ as follows: for $v:=(v_1,v_2)\in V$,
\begin{equation*}
 \pi_{ij}(v_1,v_2)=\int_{U_2}\int_{U_1}\bar{\pi}_{ij}(u_1,u_2) v_1(du_1)v_2(du_2):=\int_U \bar{\pi}_{ij}(u) v(du),
\end{equation*}
where $\ u:=(u_1,u_2)\in U$.\\
We consider a continuous time controlled  Markov chain  $Y(\cdot)$ with
 state space $S=\{1,2,\cdots \}$ and controlled rate matrix $\Pi_{v_1,v_2}=(\pi_{ij}(v_1,v_2))$, given by the stochastic integral
\begin{equation}\label{cmc2}
d Y(t) \ = \ \int_{\mathbb{R}} h(Y(t-),v_1(t),v_2(t), z) \wp(dz dt) .
\end{equation}
Here $\wp(dz dt)$ is a Poisson random measure with intensity
$m(dz)dt$, where $m(dz)$ denote the Lebesque measure on
$\mathbb{R}$. The control process $v(\cdot):=(
v_1(\cdot),v_2(\cdot))$ takes values in $V$, and  $h : S \times  V
\times \mathbb{R} \to \mathbb{R}$ is defined as follows:
\begin{equation}\label{coefficient}
h(i, v, z) \ = \
\left\{
\begin{array}{lll}
j-i & {\rm if}& z \in \Delta_{ij}(v)\\
0& &{\rm otherwise} , \\
\end{array}
\right.
\end{equation}
where $v:=(v_1, v_2)$ and  $\{ \Delta_{ij}(v) : i \neq j, \, i, j \in S\}$ denote  intervals of the form $[a, \ b)$
with length of $\Delta_{ij}(v) \ = \ \pi_{ij}(v)$ which are pairwise disjoint for each fixed $v \in  V$.\\
If $v_i(t)=\bar{v}_i(t,Y(t-))$ for some measurable
 map $\bar{v}_i : [0, \infty) \times S \to  V_i$, then $v_i(\cdot)$ is called a
 Markov strategy for the ith player. With an abuse of notation the map $\bar{v}_i$ itself is called a
 Markov strategy of player $i$. A Markov strategy $\bar{v}_i(\cdot)$ is called a stationary strategy
 if the map $\bar{v}_i$ does not  depend explicitly on time. We denote the  set of all Markov
 strategies by ${\mathcal M}_i$ and set of all stationary strategies by $\mathcal{S}_i$ for the ith player. \\
Throughout this paper we assume that:
\begin{itemize}
\item  $\bar{\pi}_{ij}(u) \geq 0$ for all $i \neq j, \ u \in  U$  and
the (infinite) matrix $(\bar{\pi}_{ij}(u))$ is conservative, i.e.,
$$ \sum_{j \in S}\bar{\pi}_{ij}(u)=0~ \mbox{for} ~i \in S ~\mbox{and}~ u \in U \, .$$
\item The function $\bar{\pi}_{ij}$ are continuous and
\[
\sup_{i \in S, u \in U} [-\bar{\pi}_{ii}(u)]:=M<\infty  \, .
\]
\end{itemize}
The existence of a unique weak solution to the equation
(\ref{cmc2}) for a pair of  Markov strategies $(v_1,v_2)$ for a
given initial distribution $\mu \in {\mathcal P}(S)$ follows using
the above assumption, see [\cite{GuoLermabook}, Theorem 2.3,
Theorem 2.5, pp.14-15].

Let $\bar{r}: S \times U_1 \times U_2 \rightarrow [0, \ \infty)$
be the running cost function. Throughout this paper, we assume
that the function $\bar{r}(\cdot)$ is nonnegative, bounded and
continuous.

We list the commonly used notations below.
\begin{itemize}
\item $ C_b[a,b]$ denotes the set of all bounded and continuous functions on $[a,b]$.
\item $ B(S)$ denotes the set of all bounded  functions on $S$.
\item $ C^1(a,b)$ denotes the set of all continuously differentiable functions on $(a,b)$.
 \item $ C^{\infty}_c(a,b)$ denotes the set of all infinitely differentiable functions on $(a,b)$ with compact support.
\item $ C_b([a,b]\times S)$ denotes the set of all functions $ f:[a,b]\times S\longrightarrow \mathbb R$ such that
 $ f(t,i)\in C_b[a,b],\;\mbox{for each}\;\; i\in S $.
 \item $ C^1((a,b)\times S)$ denotes the set of all functions $ f:(a,b)\times S\longrightarrow \mathbb R$ such that
 $ f(t,i)\in C^1(a,b),\;\mbox{for each}\;\; i\in S $.
\end{itemize}
Set
\[
 B_W(S) \ = \ \{ h : S \to \mathbb{R} | \sup_{i \in S} \frac{|h(i)|}{W(i)} < \infty \},
\]
where $W$ is the Lyapunov function as in (A1) (to be described in Section 4).
 Define for $h \in B_W(S)$,
\begin{equation*}\label{vn}
\|h\|_W \ = \ \sup_{i \in S} \frac{|h(i)|}{W(i)} \, .
\end{equation*}
 Then
$B_W(S)$ is a Banach space with the norm $\|\cdot\|_W$.  \\

\subsection{Discounted Cost Criterion}
For a pair of Markov strategies  $(v_1,v_2)$, define  $\alpha$-discounted risk-sensitive cost by
\begin{equation}\label{main}
\beta_{\alpha}^{v_1,v_2} (\theta,i)\ = \  \frac{1}{\theta} \ln  E_{i}^{v_1,v_2} \left[
 e^{ \theta \int_{0}^{\infty}e^{-\alpha t} r(Y(t),v_1 (t,Y(t-)),v_2 (t,Y(t-))) dt }  \right]
\end{equation}
for some $\theta \in (0, \ \Theta)$, and a fixed $\Theta > 0$,  $\alpha>0$ is the discount factor,
$Y(\cdot)$ is the Markov chain corresponding to $(v_1,v_2) \in  \mathcal{M}_1 \times \mathcal{M}_2 $
with $Y(0)=i$, and  $r : S \times  V \to \mathbb{R}_{+}$
is given by
\begin{equation*}
 r(i,v_1,v_2)=\int_{U_2}\int_{U_1}\bar{r}(i,u_1,u_2) v_1(du_1)v_2(du_2):=\int_U \bar{r}(i,u) v(du), \,
\end{equation*}
where  $ u:=(u_1,u_2) \ {\rm  and} \, v:=(v_1,v_2) .$\\
Let $\theta \in (0, \Theta)$ be the ``risk-sensitive parameter" chosen by the minimizer. When the state
 of the system is $i$ and players 1,2, choose  strategies $v_1 \in \mathcal{M}_1$, $v_2 \in \mathcal{M}_2$ respectively,
 the minimizer (player 1) tries to minimize his infinite-horizon discounted risk-sensitive cost $\beta_{\alpha}^{v_1,v_2} (\theta,i) $
 over his strategies whereas the maximizer (player 2) tries to maximize the same over his strategies. \\
A strategy $v_1^{*} \in \mathcal{M}_1$ is called optimal for player 1 for $(\theta,i)\in (0,\Theta)\times S$, if
\begin{equation*}\label{optimal1}
\beta_{\alpha}^{v_1^*,\tilde{v}_2} (\theta,i)\ \leq \ \sup_{v_2\in \mathcal{M}_2} \inf_{v_1\in \mathcal{M}_1}
 \beta_{\alpha}^{v_1,v_2} (\theta,i)\ := \ \displaystyle{\underline{\beta}} (\alpha,\theta,i)\; \mbox{(lower value)}
\end{equation*}
for any $\tilde{v}_2\in \mathcal{M}_2$. Similarly a strategy $v_2^{*} \in \mathcal{M}_2$ is called optimal for player 2
for $(\theta,i)\in (0,\Theta)\times S$, if
\begin{equation*}\label{optimal2}
\beta_{\alpha}^{\tilde{v}_1,v_2^*} (\theta,i)\ \geq \ \inf_{v_1\in \mathcal{M}_1} \sup_{v_2\in \mathcal{M}_2}
 \beta_{\alpha}^{v_1,v_2} (\theta,i)\ := \ \displaystyle{\overline{\beta}} (\alpha,\theta,i)\; \mbox{(upper value)}
\end{equation*}
for any $\tilde{v}_1\in \mathcal{M}_1$. The game has value if
 \begin{equation}\label{gamevalue}
\displaystyle{\underline{\beta}} (\alpha,\theta,i)\; = \ \displaystyle{\overline{\beta}}
(\alpha,\theta,i)\;= \; \displaystyle{\beta} (\alpha,\theta,i)\; \; \forall \; i\in S, \forall\; \theta \in (0,\Theta).
\end{equation}
A pair of strategies $(v_1^*,v_2^*)$ at which this value is attained is called a saddle-point
equilibrium, and then $v_1^*$ is optimal for player 1, and $v_2^*$ is optimal for player 2.
\subsection{Ergodic Cost Criterion}
For a pair of Markov strategies  $(v_1,v_2)$, the risk-sensitive ergodic cost is given by
\begin{equation}\label{main1}
\rho^{v_1,v_2} (\theta,i)\ = \   \limsup_{ T \to \infty} \frac{1}{\theta T} \ln E_{i}^{v_1,v_2}
 \Big[ e^{\theta \int^T_0  r(Y(t),v_1 (t,Y(t-)),v_2 (t,Y(t-))) dt}
 \Big] \, ,
\end{equation}
for some $\theta \in (0, \ \Theta)$, and a fixed $\Theta > 0$, $Y(\cdot)$ is the Markov chain
corresponding to $(v_1,v_2) \in  \mathcal{M}_1 \times \mathcal{M}_2 $ with $Y(0)=i$.

Optimal strategies, saddle point equilibrium, etc. for this criterion are defined analogously.
The ergodic cost  $\rho^{v_1,v_2}$ may  depend on $(\theta,i)$.

\section{Analysis of Discounted Cost Criterion}

We carry out our analysis of the discounted cost criterion via the
criterion
\begin{equation}\label{discounted_cost1}
\xi_{\alpha}^{v_1,v_2} (\theta,i)\ = \    E_{i}^{v_1,v_2} \left[
 e^{ \theta \int_{0}^{\infty}e^{-\alpha t} r(Y(t),v_1 (t,Y(t-)),v_2 (t,Y(t-))) dt }  \right].
\end{equation}
 Since logarithmic is an increasing function, therefore the optimal strategies for the criterion
 (\ref{main}) are optimal strategies for the above criterion. \\
Corresponding to the cost criterion (\ref{discounted_cost1}), the value function is defined as
\begin{equation*}
 \overline{\psi}_{\alpha}(\theta,i)= \inf_{v_1\in \mathcal{M}_1} \sup_{v_2\in \mathcal{M}_2}  \xi_{\alpha}^{v_1,v_2} (\theta,i)
\end{equation*}
and
\begin{equation*}
 \underline{\psi}_{\alpha}(\theta,i)=  \sup_{v_2\in \mathcal{M}_2} \inf_{v_1\in \mathcal{M}_1}  \xi_{\alpha}^{v_1,v_2} (\theta,i).
\end{equation*}
Using  dynamic programming heuristics, the Hamilton-Jacobi-Isaacs
(HJI) equations for discounted cost criterion are given by
\begin{eqnarray}\label{discount_hjb9}
\alpha \theta \dfrac{d\psi_\alpha}{d \theta}(\theta,i)&=& \inf_{v_1\in V_1} \sup_{v_2\in V_2}\Big [ \Pi_{v_1,v_2}
\psi_\alpha(\theta,i) +\theta r(i,v_1,v_2)\psi_\alpha(\theta,i) \Big ]  \nonumber \\
&=& \sup_{v_2\in V_2} \inf_{v_1\in V_1} \Big [ \Pi_{v_1,v_2}
\psi_\alpha(\theta,i) +\theta r(i,v_1,v_2)\psi_\alpha(\theta,i) \Big ] \nonumber \\
\displaystyle{ \psi_{\alpha}(0,i) } &=& 1 ,
\end{eqnarray}
where $\Pi_{v_1,v_2} f(i):= \displaystyle{\sum_{j\in S}} \pi_{ij}(v_1,v_2)f(j)$, for any function $f(i)$.\\
Next we prove that the equations (\ref{discount_hjb9}) have a smooth, bounded solution.\\
Fix $\epsilon > 0$ and consider the ordinary differential equation
(ODE)
\begin{eqnarray}\label{hjb12}
\alpha \theta \dfrac{d\psi^\epsilon_\alpha}{d \theta}(\theta,i)&=& \inf_{v_1\in V_1} \sup_{v_2\in V_2}\Big [ \Pi_{v_1,v_2}
\psi^\epsilon_\alpha(\theta,i) +\theta r(i,v_1,v_2)\psi^\epsilon_\alpha(\theta,i) \Big ]  \nonumber \\
&=& \sup_{v_2\in V_2} \inf_{v_1\in V_1} \Big [ \Pi_{v_1,v_2}
\psi^\epsilon_\alpha(\theta,i) +\theta r(i,v_1,v_2)\psi^\epsilon_\alpha(\theta,i) \Big ]
\end{eqnarray}
$$\psi^\epsilon_\alpha(\epsilon,i)=e^{\frac{\epsilon}{\alpha}\|r\|_{\infty}}:=h_\epsilon,$$
where $\|\cdot\|_{\infty}$ denotes the supnorm. Note that the second equality follows from Fan's minimax theorem, see [\cite{Fan}, Theorem 3].\\
Let $\delta >0$.
 Define the nonlinear operator $T : C_b([\epsilon, \epsilon+\delta] \times S) \rightarrow C_b([\epsilon, \epsilon+\delta] \times S)$
by
\begin{equation*}
T f(\eta,i):= e^{\frac{\epsilon}{\alpha}\|r\|_{\infty}}+ \frac{1}{ \alpha}\int_{\epsilon}^{\eta}\inf_{v_1\in V_1} \sup_{v_2\in V_2}
\Big [ \frac{1}{\theta }\Pi_{v_1,v_2} f(\theta,i)  + r (i,v_1,v_2)f(\theta,i) \Big ]d \theta.
\end{equation*}
By using the fact $\displaystyle{\sup_{i \in S, u \in U}} [-\bar{\pi}_{ii}(u)]=M<\infty $ and  $r$ is bounded, we have
 $$\| T f_1- T f_2\|_{\infty} \leq \frac{1}{ \alpha}\Big [ \|r\|_{\infty}\delta
+2M \ln \Big(1+ \frac{\delta}{\epsilon}\Big )\Big ] \|f_1-f_2\|_{\infty}. $$
 Choose $\delta$ such that $\frac{1}{ \alpha}\Big [ \|r\|_{\infty}\delta
+2M \ln \Big(1+ \frac{\delta}{\epsilon}\Big )\Big ] <1$. Then $T$ is a contraction operator.
 Therefore by Banach's fixed point theorem there exists a function $\psi^\epsilon_\alpha \in C_b([\epsilon, \epsilon+\delta] \times S)$ such that 
 \begin{equation*}
\psi^\epsilon_\alpha(\eta,i)= e^{\frac{\epsilon}{\alpha}\|r\|_{\infty}}+ \frac{1}{ \alpha}\int_{\epsilon}^{\eta}\inf_{v_1\in V_1} \sup_{v_2\in V_2}
\Big [ \frac{1}{\theta }\Pi_{v_1,v_2} \psi^\epsilon_\alpha(\theta,i)  + r (i,v_1,v_2)\psi^\epsilon_\alpha(\theta,i) \Big ]d \theta.
\end{equation*}
Note that the bracketed term in the above integrand is bounded and jointly continuous in $(\theta, v_1,v_2)$. Since $V_1$ and $V_2$ are compact metric spaces, it follows that the integrand above is bounded and continuous in $\theta \in [ \epsilon, \epsilon+\delta]$.
Thus it follows that $\psi^\epsilon_\alpha $ is in
$C^1(( \epsilon, \epsilon+\delta] \times S) \cap C_b([\epsilon, \epsilon + \delta] \times S)$. Proceeding in this way we get a
$C^1( (\epsilon, \Theta) \times S) \cap C_b ([\epsilon, \Theta) \times S)$  solution for the  ODE (\ref{hjb12}). Let
\begin{equation*}
\bar{v}_i:(0,\Theta)\times S \to V_i, \; \; i=1,2,
\end{equation*}
be measurable functions such that
\begin{eqnarray}\label{sup1}
&& \inf_{v_1\in V_1} \sup_{v_2\in V_2}\Big [ \Pi_{v_1,v_2}
\psi^\epsilon_\alpha(\theta,i) +\theta r(i,v_1,v_2)\psi^\epsilon_\alpha(\theta,i) \Big ]  \nonumber \\
&=& \sup_{v_2\in V_2}  \Big [ \Pi_{\bar{v}_1,v_2}
\psi^\epsilon_\alpha(\theta,i) +\theta r(i,\bar{v}_1(\theta,i),v_2)\psi^\epsilon_\alpha(\theta,i) \Big ]
\end{eqnarray}
and
\begin{eqnarray}\label{sup2}
&&\sup_{v_2\in V_2} \inf_{v_1\in V_1} \Big [ \Pi_{v_1,v_2}
\psi^\epsilon_\alpha(\theta,i) +\theta r(i,v_1,v_2)\psi^\epsilon_\alpha(\theta,i) \Big ]  \nonumber \\
&=& \inf_{v_1\in V_1}  \Big [ \Pi_{v_1,\bar{v}_2}
\psi^\epsilon_\alpha(\theta,i) +\theta r(i,,v_1,\bar{v}_2(\theta,i))\psi^\epsilon_\alpha(\theta,i) \Big ] .
\end{eqnarray}
The existence of such measurable maps are ensured by Bene$\check{s}$ measurable selection theorem,
see \cite{Benes}.
Let
\begin{equation*}
v_i^*:\mathbb{R}_+\times S \to V_i, \; \; i=1,2,
\end{equation*}
be defined by
\begin{equation*}
v_i^*(t,i)=\bar{v}_i(\theta e^{-\alpha t},i), \; \; i=1,2.
\end{equation*}
 Set $\theta(t)=\theta e^{-\alpha t}$ and define $T_{\epsilon}$ by
$$T_\epsilon=\inf \{t\geq 0: \theta(t) =\epsilon\}.$$
For $(v^*_1,v_2) \in \mathcal{M}_1 \times \mathcal{M}_2 $, applying It$\hat{\rm o}$ formula (see \cite{GuoLermabook}, Appendix C, pp. 218-219) to the
function $$ e^{\int_0^{t} \theta(s) r(Y(s),v_1^*(s,Y(s-)),v_2(s,Y(s-)))ds}\psi^\epsilon_\alpha (\theta(t), Y(t)),$$
 we obtain
\begin{eqnarray*}
 & & E_{i}^{v_1,v_2}[ e^{\int_0^{T_\epsilon} \theta(s) r(Y(s),v^*_1(s,Y(s-)),v_2(s,Y(s-)))ds}\psi^\epsilon_\alpha (\theta(T_\epsilon),
Y({T_\epsilon}))]-\psi^\epsilon_\alpha(\theta,i) \\
 && = \ E_{i}^{v_1,v_2} \Big [ \int_0^{T_\epsilon} e^{\int_0^t \theta(s) r(Y(s),v^*_1(s,Y(s-)),v_2(s,Y(s-)))ds}
 \Big \{-\alpha \theta(t)  \dfrac{d\psi^\epsilon_\alpha}{d \theta}(\theta(t),Y(t)) \\
&& + \Pi_{v_1,v_2} \psi^\epsilon_\alpha(\theta(t),Y(t))
 \ + \theta(t) r(Y(t),v^*_1(t,Y(t-)),v_2(t,Y(t-)))\psi^\epsilon_\alpha(\theta(t),Y(t)) \Big \}dt \Big].
\end{eqnarray*}
Since $\psi^\epsilon_\alpha$ satisfies (\ref{sup1}), we obtain
\begin{eqnarray*}
 & & E_{i}^{v_1^*,v_2}[ e^{\int_0^{T_\epsilon} \theta(s) r(Y(s),v_1^*(s,Y(s-)),v_2(s,Y(s-)))ds} h_{\epsilon}]
-\psi^\epsilon_\alpha(\theta,i) \leq 0 ,
\end{eqnarray*}
where $h_\epsilon$ is as in (\ref{hjb12}).
Since $v_2$ is arbitrary, we get
\begin{eqnarray}\label{inf1}
 \psi^\epsilon_\alpha(\theta,i)&\geq & \sup_{v_2 \in \mathcal{M}_2 }  E_{i}^{v_1^*,v_2} \left[ h_\epsilon
 e^{\int_0^{T_\epsilon} \theta(s) r(Y(s),v_1^*(s,Y(s-)),v_2(s,Y(s-)))ds}  \right] .
\end{eqnarray}
 Using analogous arguments, we can show that
\begin{eqnarray}\label{inf2}
 \psi^\epsilon_\alpha(\theta,i)&\leq & \inf_{v_1 \in \mathcal{M}_1 }  E_{i}^{v_1,v_2^*} \left[ h_\epsilon
 e^{\int_0^{T_\epsilon} \theta(s) r(Y(s),v_1(s,Y(s-)),v_2^*(s,Y(s-)))ds}  \right].
\end{eqnarray}
Therefore, from (\ref{inf1}) and (\ref{inf2}), we obtain
 \begin{eqnarray}\label{vf}
  \psi^\epsilon_\alpha(\theta,i)&=\displaystyle{\sup_{v_2 \in \mathcal{M}_2 }
  \inf_{v_1 \in \mathcal{M}_1 }}  E_{i}^{v_1,v_2} \left[ h_\epsilon
 e^{\int_0^{T_\epsilon} \theta(s) r(Y(s),v_1(s,Y(s-)),v_2(s,Y(s-)))ds}  \right]\nonumber \\
 &=\displaystyle{ \inf_{v_1 \in \mathcal{M}_1 } \sup_{v_2 \in \mathcal{M}_2 } } E_{i}^{v_1,v_2} \left[ h_\epsilon
 e^{\int_0^{T_\epsilon} \theta(s) r(Y(s),v_1(s,Y(s-)),v_2(s,Y(s-)))ds}  \right].
\end{eqnarray}
Next we take limit  of $ \psi^\epsilon_\alpha$ as $\epsilon \to 0$ and
prove that the limit function satisfies (\ref{discount_hjb9}), i.e., we prove the following theorem.
\begin{theorem}\label{th1}
 There exists a unique solution $\psi_{\alpha}$ in the class $C_b((0,\Theta)\times S)
  \cap C^1((0,\Theta)\times S)$ to (\ref{discount_hjb9}).
The solution admits the following representation
 \begin{eqnarray*}\label{vf1}
  \psi_\alpha(\theta,i) &=& \sup_{v_2 \in \mathcal{M}_2 }
  \inf_{v_1 \in \mathcal{M}_1 }  E_{i}^{v_1,v_2} \left[  e^{\theta \int_0^{\infty}
  e^{-\alpha s} r(Y(s),v_1(s,Y(s-)),v_2(s,Y(s-)))ds}  \right]\nonumber \\
 &=& \inf_{v_1 \in \mathcal{M}_1 } \sup_{v_2 \in \mathcal{M}_2 }  E_{i}^{v_1,v_2}
 \left[ e^{\theta \int_0^{\infty} e^{-\alpha s} r(Y(s),v_1(s,Y(s-)),v_2(s,Y(s-)))ds}  \right].
\end{eqnarray*}
Furthermore $\psi_\alpha$ is the value function for the discounted cost criterion (\ref{discounted_cost1}).
 Moreover, a saddle point equilibrium exists in $\mathcal{M}_1 \times \mathcal{M}_2$.
\end{theorem}
\begin{proof}
First recall the stochastic representation of $\psi_{\alpha}^{\epsilon}$  from  (\ref{vf}),
 \begin{eqnarray*}
  \psi^\epsilon_\alpha(\theta,i) &=& \sup_{v_2 \in \mathcal{M}_2 }
  \inf_{v_1 \in \mathcal{M}_1 }  E_{i}^{v_1,v_2} \left[ h_\epsilon
 e^{\int_0^{T_\epsilon} \theta(s) r(Y(s),v_1(s,Y(s-)),v_2(s,Y(s-)))ds}  \right]\nonumber \\
 &=& \inf_{v_1 \in \mathcal{M}_1 } \sup_{v_2 \in \mathcal{M}_2 }  E_{i}^{v_1,v_2} \left[ h_\epsilon
 e^{\int_0^{T_\epsilon} \theta(s) r(Y(s),v_1(s,Y(s-)),v_2(s,Y(s-)))ds}  \right].
\end{eqnarray*}
From the representation of $\psi_{\alpha}^{\epsilon}$, we have
$$ 1\leq \psi_{\alpha}^{\epsilon}(\theta,i)\leq h_\epsilon e^{\frac{\theta }{\alpha}\|r\|_{\infty}(1- e^{-\alpha T_\epsilon})} = e^{\frac{\theta }{\alpha}\|r\|_{\infty}}$$
for every $\epsilon >0$, and all $(\theta,i)$.\\
By closely mimicking the arguments in the proof of  [\cite{GhoshSaha}, Theorem 3.4], it follows that the HJB equation (\ref{discount_hjb9}) has a  solution in $C_b((0,\Theta)\times S)
  \cap C^1((0,\Theta)\times S)$.
Let
\begin{equation*}
\bar{v}_i:(0,\Theta)\times S \to V_i, \; \; i=1,2,
\end{equation*}
be measurable selectors such that
\begin{eqnarray}\label{sup11}
&& \inf_{v_1\in V_1} \sup_{v_2\in V_2}\Big [ \Pi_{v_1,v_2}
\psi_\alpha(\theta,i) +\theta r(i,v_1,v_2)\psi_\alpha(\theta,i) \Big ]  \nonumber \\
&=& \sup_{v_2\in V_2}  \Big [ \Pi_{\bar{v}_1,v_2}
\psi_\alpha(\theta,i) +\theta r(i,\bar{v}_1(\theta,i),v_2)\psi_\alpha(\theta,i) \Big ]
\end{eqnarray}
and
\begin{eqnarray}\label{sup21}
&&\sup_{v_2\in V_2} \inf_{v_1\in V_1} \Big [ \Pi_{v_1,v_2}
\psi_\alpha(\theta,i) +\theta r(i,v_1,v_2)\psi_\alpha(\theta,i) \Big ]  \nonumber \\
&=& \inf_{v_1\in V_1}  \Big [ \Pi_{v_1,\bar{v}_2}
\psi_\alpha(\theta,i) +\theta r(i,,v_1,\bar{v}_2(\theta,i))\psi_\alpha (\theta,i) \Big ] .
\end{eqnarray}
 Let
\begin{equation*}
v_i^*:\mathbb{R}_+\times S \to V_i, \; \; i=1,2,
\end{equation*}
be defined by
\begin{equation}\label{selector}
v_i^*(t,i)=\bar{v}_i(\theta e^{-\alpha t},i), \; \; i=1,2.
\end{equation}
For $(v_1^*,v_2) \in \mathcal{M}_1 \times \mathcal{M}_2 $, applying It$\hat{\rm o}$ formula to the function
 $$ e^{\int_0^{t} \theta(s) r(Y(s),v_1^*(s,Y(s-)),v_2(s,Y(s-)))ds}\psi_\alpha (\theta(t), Y(t))$$
and using (\ref{sup11}), we get
\begin{eqnarray*}
 & & E_{i}^{v_1^*,v_2}[ e^{\int_0^{T} \theta(s) r(Y(s),v_1^*(s,Y(s-)),v_2(s,Y(s-)))ds}\psi_\alpha (\theta(T),
Y(T))]-\psi_\alpha(\theta,i) \leq 0.
\end{eqnarray*}
Since $v_2$ is arbitrary, we get
\begin{eqnarray*}
 \psi_\alpha(\theta,i)&\geq & \sup_{v_2 \in \mathcal{M}_2 }  E_{i}^{v_1^*,v_2} \left[ \psi_\alpha (\theta(T),
Y(T)) e^{\int_0^{T} \theta(s) r(Y(s),v_1^*(s,Y(s-)),v_2(s,Y(s-)))ds}  \right].
\end{eqnarray*}
Since $ 1 \leq \bar{\psi}^\epsilon_{\alpha}
\leq e^{\frac{\theta}{\alpha} \|r\|_{\infty}}$ for all $\epsilon > 0$, we get
\begin{eqnarray*}
 \psi_\alpha(\theta,i)&\geq & \sup_{v_2 \in \mathcal{M}_2 }  E_{i}^{v_1^*,v_2} \left[
 e^{\int_0^{T} \theta(s) r(Y(s),v_1^*(s,Y(s-)),v_2(s,Y(s-)))ds}  \right].
\end{eqnarray*}
By using monotone convergence theorem for  letting $T\rightarrow \infty$ in the above we obtain
\begin{eqnarray}\label{inf11}
 \psi_\alpha(\theta,i)&\geq & \sup_{v_2 \in \mathcal{M}_2 }  E_{i}^{v_1^*,v_2}
  \left[  e^{\int_0^{\infty} \theta(s) r(Y(s),v_1^*(s,Y(s-)),v_2(s,Y(s-)))ds}  \right].
\end{eqnarray}
Using analogous arguments  we can show that
\begin{eqnarray}\label{inf21}
 \psi_\alpha(\theta,i)&\leq & \inf_{v_1 \in \mathcal{M}_1 }  E_{i}^{v_1,v_2^*}
  \left[ e^{\int_0^{\infty} \theta(s) r(Y(s),v_1(s,Y(s-)),v_2^*(s,Y(s-)))ds}  \right].
\end{eqnarray}
Therefore, from (\ref{inf11}) and (\ref{inf21}), we obtain
 \begin{eqnarray*}
  \psi_\alpha(\theta,i) &=& \sup_{v_2 \in \mathcal{M}_2 }
  \inf_{v_1 \in \mathcal{M}_1 }  E_{i}^{v_1,v_2} \left[
 e^{\int_0^{\infty} \theta(s) r(Y(s),v_1(s,Y(s-)),v_2(s,Y(s-)))ds}  \right]\nonumber \\
 &=& \inf_{v_1 \in \mathcal{M}_1 } \sup_{v_2 \in \mathcal{M}_2 }  E_{i}^{v_1,v_2} \left[
 e^{\int_0^{\infty} \theta(s) r(Y(s),v_1(s,Y(s-)),v_2(s,Y(s-)))ds}  \right].
\end{eqnarray*}
It is easy to check that $\psi_\alpha$ is the value function for the discounted cost criterion
(\ref{discounted_cost1}). Moreover, the pair of Markov strategies given by (\ref{selector}) forms a saddle point equilibrium.
This completes the proof.
\end{proof}
\section{Analysis of Ergodic Cost Criterion}
In this section  we prove the existence of value and stationary Markov saddle point strategies for
the ergodic cost criterion under the following assumption:\\
\noindent {\bf (A1)}(Lyapunov condition) There exist constants $b
> 0, \ \delta >0$, a finite set $C$ and a map $W : S \to
[1, \infty)$ with $W(i)\to \infty$ as $i \to \infty$, such
that
\begin{eqnarray*}
\Pi_v W(i) &\leq&  -2 \delta W(i) + b I_C(i), \ i \in S, \ v \in V .
 \end{eqnarray*}

 Throughout this section, we assume that
for every pair of stationary  Markov strategies  $(v_1,v_2)$ the corresponding Markov chain is irreducible.

 We carry out our analysis of the ergodic cost criterion as a limit of the corresponding finite horizon cost criterion given by
\begin{equation}\label{finitecost}
I_{T}(i,v_1, v_2):=   E_{i}^{v_1,v_2} \left[ e^{ \int_{0}^{T}
r(Y(s),v_1(s,Y(s-)),v_2(s,Y(s-))) ds }
 \right]
\end{equation}
where $Y(\cdot)$ is the Markov chain corresponding to $(v_1,v_2) \in \mathcal{M}_1 \times \mathcal{M}_2 $ with initial condition $i \in S$.
Using the dynamic programming heuristics, the HJI equations for the above cost criterion, are given by
\begin{eqnarray}\label{fh_hjb9}
- \dfrac{d\phi}{dt}(t,i)&=& \inf_{v_1\in V_1} \sup_{v_2\in V_2}\Big [ \Pi_{v_1,v_2}
\phi(t,i) + r(i,v_1,v_2)\phi(t,i) \Big ]  \nonumber \\
&=& \sup_{v_2\in V_2} \inf_{v_1\in V_1} \Big [ \Pi_{v_1,v_2}
\phi(t,i) + r(i,v_1,v_2)\phi(t,i) \Big ] \nonumber \\
\displaystyle{ \phi(T,i) } &=& 1 .
\end{eqnarray}
 As before, we can show the existence of a $C^1((0,T) \times S)\cap C_b([0,T] \times S)$  solution for the  ODE (\ref{fh_hjb9}).
Using a standard application of It$\hat{\rm o}$'s formula we get
  \begin{eqnarray*}
 \phi(t,i) &=& \sup_{v_2 \in \mathcal{M}_2 }
  \inf_{v_1 \in \mathcal{M}_1 }  E^{v_1,v_2}_i \left[ e^{  \int_{t}^{T} r(Y(s),v_1(s,Y(s-)),v_2(s,Y(s-))) ds } \right]\nonumber \\
 &=& \inf_{v_1 \in \mathcal{M}_1 } \sup_{v_2 \in \mathcal{M}_2 }  E^{v_1,v_2}_i \left[ e^{  \int_{t}^{T} r(Y(s),v_1(s,Y(s-)),v_2(s,Y(s-))) ds }  \right].
\end{eqnarray*}
Set $ \psi(t,i)=\phi(T-t,i)$. Then $\psi$ is the unique $C^1((0,T) \times S)\cap C_b([0,T] \times S)$  solution to
\begin{eqnarray*}\label{fh_hjb2}
 \dfrac{d\psi}{dt}(t,i)&=& \inf_{v_1\in V_1} \sup_{v_2\in V_2}\Big [ \Pi_{v_1,v_2}
\psi(t,i) + r(i,v_1,v_2)\psi(t,i) \Big ]  \nonumber \\
&=& \sup_{v_2\in V_2} \inf_{v_1\in V_1} \Big [ \Pi_{v_1,v_2}
\psi(t,i) + r(i,v_1,v_2)\psi(t,i) \Big ] \nonumber \\
\displaystyle{ \psi(0,i) } &=& 1 .
\end{eqnarray*}
Using It$\hat{\rm o}$'s formula, we obtain
  \begin{eqnarray*}
 \psi(t,i) &=& \sup_{v_2 \in \mathcal{M}_2 }
  \inf_{v_1 \in \mathcal{M}_1 }  E_{i}^{v_1,v_2} \left[ e^{  \int_{0}^{t} r(Y(s),v_1(s,Y(s-)),v_2(s,Y(s-))) ds }  \right]\nonumber \\
 &=& \inf_{v_1 \in \mathcal{M}_1 } \sup_{v_2 \in \mathcal{M}_2 }  E_{i}^{v_1,v_2} \left[ e^{  \int_{0}^{t} r(Y(s),v_1(s,Y(s-)),v_2(s,Y(s-))) ds }  \right].
\end{eqnarray*}
Formally, using separation of variables, we write $$\psi(t,i)=e^{\rho t}\hat{\psi}(i). $$
This yields
\begin{eqnarray}\label{brs}
 \rho ~\hat{\psi}(i) &=& \inf_{v_1\in V_1} \sup_{v_2\in V_2}\Big [ \Pi_{v_1,v_2}
\hat\psi(i) + r(i,v_1,v_2)\hat\psi(i) \Big ]  \nonumber \\
&=& \sup_{v_2\in V_2} \inf_{v_1\in V_1} \Big [ \Pi_{v_1,v_2}
\hat\psi(i) + r(i,v_1,v_2)\hat\psi(i) \Big ].
\end{eqnarray}
The above equation is the HJI equation for the ergodic cost (\ref{main1}).

 We now proceed to make a rigorous analysis of the above. First
we truncate our cost function which plays a crucial role  to
derive the HJI equations and find the value of the game. Let $r_n
: S \times  V \to [0, \ \infty)$ be given by
\begin{equation}
r_{n}:= \left\{
\begin{array}{ll}
r & {\rm if} \;  i\in \{1,2,\cdots,n\}\\
0 & \rm{otherwise}   \\
\end{array} \right. \label{(R)}
\end{equation}
and
 \begin{eqnarray*}
 \psi^n(t,i) &=& \sup_{v_2 \in \mathcal{M}_2 }
  \inf_{v_1 \in \mathcal{M}_1 }  E_{i}^{v_1,v_2} \left[ e^{  \int_{0}^{t} r_n(Y(s),v_1(s,Y(s-)),v_2(s,Y(s-))) ds }  \right]\nonumber \\
 &=& \inf_{v_1 \in \mathcal{M}_1 } \sup_{v_2 \in \mathcal{M}_2 }  E_{i}^{v_1,v_2} \left[ e^{  \int_{0}^{t} r_n(Y(s),v_1(s,Y(s-)),v_2(s,Y(s-))) ds }  \right].
\end{eqnarray*}
Then, as above, we can show that $\psi^n$ is the unique solution in  $C^1((0,T) \times S)\cap C_b([0,T] \times S)$  to
\begin{eqnarray*}\label{fh_hjb2-r_n}
 \dfrac{d\psi^n}{dt}(t,i)&=& \inf_{v_1\in V_1} \sup_{v_2\in V_2}\Big [ \Pi_{v_1,v_2}
\psi^n(t,i) + r_n(i,v_1,v_2)\psi_n(t,i) \Big ]  \nonumber \\
&=& \sup_{v_2\in V_2} \inf_{v_1\in V_1} \Big [ \Pi_{v_1,v_2}
\psi^n(t,i) + r_n(i,v_1,v_2)\psi^n(t,i) \Big ] \nonumber \\
\displaystyle{ \psi^n(0,i) } &=& 1 .
\end{eqnarray*}
Now onward, we fix a reference state $i_0 \in S$ such that $W(i_0)\geq 1+\frac{b}{\delta}$ and set
\begin{equation}\label{evf1}
\bar{\psi}^n(t,i)=\dfrac{\psi^n(t,i)}{\psi^n(t,i_0)}.
\end{equation}
Then it is easy to see that $\bar{\psi}^n$
is the unique solution in $C^1((0,T) \times S)\cap C_b([0,T] \times S)$ to
\begin{eqnarray}\label{fh_hjb7}
  \dfrac{d\bar{\psi}^n}{dt}(t,i)+\dfrac{\bar{\psi}^n(t,i)}{\psi^n(t,i_0)}
  \dfrac{d\psi^n}{dt}(t,i_0)&=& \inf_{v_1\in V_1} \sup_{v_2\in V_2}\Big [ \Pi_{v_1,v_2}
\bar\psi^n(t,i) + r_n(i,v_1,v_2)\bar\psi^n(t,i) \Big ]  \nonumber \\
&=& \sup_{v_2\in V_2} \inf_{v_1\in V_1} \Big [ \Pi_{v_1,v_2}
\bar\psi^n(t,i) + r_n(i,v_1,v_2)\bar\psi^n(t,i) \Big ] \nonumber \\
\displaystyle{ \bar\psi^n(0,i) } &=& 1 .
\end{eqnarray}
Next we take limit as $t \to \infty $ in (\ref{fh_hjb7}), to derive the existence of a solution for
ergodic HJI equation with cost function $r_n$. For this  we want to show that $\bar{\psi}^n(t,i)$ is
uniformly bounded (for each fixed $n$ ). To this end fix a strategy of player 2 and consider the
corresponding optimal control problem for player 1 .\\
 Let $v_{2n}^*:\mathbb{R}_{+}\times S \to V_2$ be a measurable map such that
\begin{eqnarray*}
&& \inf_{v_1\in V_1} \sup_{v_2\in V_2}\Big [ \Pi_{v_1,v_2}
\psi^n(t,i) + r_n(i,v_1,v_2)\psi^n(t,i) \Big ]  \nonumber \\
&=& \inf_{v_1\in V_1} \Big [ \Pi_{v_1,v_{2n}^*(t,i)}
\psi^n(t,i) + r_n(i,v_1,v_{2n}^*(t,i))\psi^n(t,i) \Big ].
\end{eqnarray*}
We suppress the dependence of $n$ on $v_{2n}^*$ and write $v_2^*$ instead.

For the fixed Markov strategy $v_2^* \in \mathcal{M}_2$  consider the pure jump processes  given by
\begin{equation}\label{cmc1}
d Y_{v_2^*}(t) \ = \ \int_{\mathbb{R}} h(Y_{v_2^*}(t-),v_1(t),v_2^*(t, Y_{v_2^*}(t-)), z) \wp(dz dt) ,
\end{equation}
where $h$ is as in (\ref{coefficient}). \\
Now we consider a new auxiliary continuous time Markov decision problem (CTMDP) corresponding to the process (\ref{cmc1}),
 i.e., player 2 fixes the strategy $v_2^*$ and player 1 treats it as a CTMDP.
 First we define the set of all admissible controls denoted by $\mathcal{A}$.

A $V_1$-valued process $v_1(\cdot)$ is said to be admissible if it is predictable
and the equation
\begin{equation}\label{cmc}
d Y_{v_2^*}(t) \ = \ \int_{\mathbb{R}} h(Y_{v_2^*}(t-), v_1(t),v_2^*(t, Y_{v_2^*}(t-)), z) \wp(dz dt)
\end{equation}
has a unique weak  solution for each initial $Y_0$ independent of $\wp(dz dt)$.

For an admissible control $v_1(\cdot)\in \mathcal{A}$, the risk-sensitive cost for the finite horizon [0,T] is defined by
\begin{equation*}\label{ergodicriskcost1}
I^n_{v_2^*}(i, v_1(\cdot)) \ = \  E_{i}^{v_1,v_2^*} \Big[ e^{\int^T_0 r_n(Y_{v_2^*}(t), v_1(t),
v_2^*(t, Y_{v_2^*}(t-))) dt}   \Big] \, ,
\end{equation*}
where $Y_{v_2^*}(\cdot)$ is the pure jump process (\ref{cmc}) corresponding to $v_1(\cdot)$ and  initial condition $i \in S$.\\
Consider the ODE
\begin{eqnarray}\label{discount_hjb4}
- \dfrac{d\Psi^n_{v_2^*}}{dt}(t,i)&=& \inf_{v_1\in V_1} \Big [ \Pi_{v_1,v_2^*(t,i)}
\Psi^n_{v_2^*}(t,i) + r_n(i,v_1,v_2^*(t,i))\Psi^n_{v_2^*}(t,i) \Big ]  \nonumber \\
\displaystyle{ \Psi^n_{v_2^*}(T,i) } &=& 1 .
\end{eqnarray}
 Define the nonlinear operator $\mathcal{T} : C_b([0, T] \times S) \rightarrow C_b([0, T] \times S)$
by
\begin{equation*}
\mathcal {T} f(t,i):= 1+ \int_{t}^{T}\inf_{v_1\in V_1}
\Big [\Pi_{v_1,v_2^*(t,i)} f(s,i)  + r_n (i,v_1,v_2^*(t,i))f(s,i) \Big ]ds.
\end{equation*}
 As before, we get the existence of a solution to (\ref{discount_hjb4}) in $ C_b([0, T] \times S) $.
 Using a standard application of It$\hat{\rm o}$'s formula, we obtain
\begin{eqnarray*}\label{vf3}
 \Psi^n_{v_2^*}(t,i) &=&   \inf_{v_1(\cdot) \in \mathcal{A} }  E_{i}^{v_1,v_2^*} \left[ e^{  \int_{t}^{T}
 r_n(Y_{v_2^*}(s),v_1(s,Y_{v_2^*}(s-)),v_2^*(s, Y_{v_2^*}(s-))) ds } \right].
\end{eqnarray*}
Set $ \psi^n_{v_2^*}(t,i)=\Psi^n_{v_2^*}(T-t,i)$. Then $\psi^n_{v_2^*}$ is the unique solution in $ C_b([0, T] \times S) $  to
\begin{eqnarray}\label{fh_hjb5}
 \dfrac{d\psi^n_{v_2^*}}{dt}(t,i)&=& \inf_{v_1\in V_1}\Big [ \Pi_{v_1,v_2^*(t,i)}
\psi^n_{v_2^*}(t,i) + r_n(i,v_1,v_2^*(t,i))\psi^n_{v_2^*}(t,i) \Big ]  \nonumber \\
\displaystyle{ \psi_{v_2^*}(0,i) } &=& 1 .
\end{eqnarray}
Using It$\hat{\rm o}$'s formula, we obtain
  \begin{eqnarray*}\label{vf3}
 \psi^n_{v_2^*}(t,i) &=&   \inf_{v_1(\cdot) \in \mathcal{A} }  E_{i}^{v_1,v_2^*} \left[
e^{  \int_{0}^{t} r_n(Y_{v_2^*}(s),v_1(s),v_2^*(s, Y_{v_2^*}(s-))) ds }  \right].
 \end{eqnarray*}
It is easy to see that any minimizing selector in (\ref{fh_hjb5}) corresponding to $\psi^n_{v_2^*}$
is optimal for the
finite horizon CTMDP for player 1. Since any minimizing selector corresponds to a Markov control, we have
\[
\psi^n_{v_2^*}(t,i)= \inf_{v_1 \in \mathcal{M}_1}  E_{i}^{v_1,v_2^*} \left[ e^{ \int_{0}^{t} r_n(Y_{v_2^*}(s),
v_1(s, Y_{v_2^*}(s-)),v_2^*(s, Y_{v_2^*}(s-))) ds }  \right].
\]
For the reference state $i_0 \in S$, we define
 $$\bar{\psi}^n_{v_2^*}(t,i)=\dfrac{\psi^n_{v_2^*}(t,i)}{\psi^n_{v_2^*}(t,i_0)}.$$
A simple calculation  shows that
 $\bar{\psi}^n_{v_2^*}=\bar{\psi}^n$, where $\bar{\psi}^n$ is as in (\ref{evf1}).

By closely mimicking the arguments in [\cite{SP}, Theorem 3.1], one can easily get the following
multiplicative DPP; we omit the details.
\begin{theorem}\label{thmdpp} For any set $\tilde{S} \subseteq S,$
\[
\psi^n_{v_2^*}(t, i) \ = \, \inf_{v_1(\cdot) \in {\mathcal A}} E_{i}^{v_1,v_2^*} \Big[ e^{\int^{t \wedge
\tau}_0 r_n(Y_{v_2^*}(s), v_1(s),v_2^*(s,Y_{v_2^*}(s-))) ds }
\psi^n_{v_2^*}(t- (t \wedge \tau), Y_{v_2^*}(t \wedge \tau)) \Big], \ t \geq 0 \, ,
\]
where $\tau$ is the hitting time of the process $Y_{v_2^*}(\cdot)$ to the set $\tilde{S}$.
\end{theorem}

Using the similar arguments as in the proofs of \cite{SP1}, we can prove the following results; we omit the details.
\begin{lemma}\label{expbd}
 Assume (A1). Let $Y(\cdot)$ be a process (\ref{cmc2}) corresponding to $(v_1,v_2^*) \in {\mathcal M}_1 \times \mathcal{M}_2$.  Then
\begin{eqnarray*}
 E_{i}^{v_1,v_2^*} \left [e^{\delta \tau_{i_0}}  \right] &\leq &  W(i), \, i \in S   ,
\end{eqnarray*}
where $\tau_{i_0} =\inf \{t\geq 0:Y(t)=i_0\}$.
\end{lemma}
\begin{lemma}\label{bd11} Assume (A1) and $\|r\|_{\infty} \leq \delta$, where $\delta >0$ is given in (A1).
Then $$|\bar{\psi}^n_{v_2^*}(t, i)| \, \leq \,  W(i) , \ t \geq 0,
i \in S.$$
\end{lemma}
\begin{lemma}\label{bd12} Assume (A1) and $\|r\|_{\infty} \leq \delta$, where $\delta >0$ is given in (A1).
Then $$\sup_{t > 0, i\in S} \|\bar{\psi}_{v_2^*}^n(t, i)\|_{\infty} < \infty.$$
\end{lemma}
\begin{proof}
Let $i\geq n+1$ and $Y_{v_2^*}(\cdot)$ be the solution corresponding to
 $\hat{v}(\cdot)\in \mathcal{M}_1$ with initial condition $i$. Then from Theorem \ref{thmdpp}, we have
\[
\begin{array}{lll}
\psi_{v_2^*}^n(t, i) & \leq & \displaystyle{ E_{i}^{\hat{v},v_2^*} \Big[ e^{\int^{t \wedge \tau }_0 r_n(Y_{v_2^*}(s), \hat{v}(s,Y_{v_2^*}(s)),{v_2^*}(s,Y_{v_2^*}(s))) ds}
\psi_{v_2^*}^n(t - (t \wedge \tau), Y_{v_2^*}(t \wedge \tau)) \Big] \, ,}\\
& = & \displaystyle{ E_{i}^{\hat{v},v_2^*} \Big[ \psi_{v_2^*}^n(t - (t \wedge \tau), Y_{v_2^*}(t \wedge \tau)) I \{ t \leq \tau\}\Big] }\\
&& \displaystyle{ + E_{i}^{\hat{v},v_2^*} \Big[ \psi_{v_2^*}^n(t - (t \wedge \tau), Y_{v_2^*}(t \wedge \tau))I \{ t > \tau\} \Big] }\\
& = & \displaystyle{ E_{i}^{\hat{v},v_2^*} \Big[ \psi_{v_2^*}^n(0, Y_{v_2^*}(t)) I \{ t \leq \tau\}\Big] }\\
&& \displaystyle{ + E_{i}^{\hat{v},v_2^*} \Big[ \psi_{v_2^*}^n(t -  \tau, Y_{v_2^*}(\tau))I \{ t > \tau\} \Big] }\\
& \leq & 1 \displaystyle{ + E_{i}^{\hat{v},v_2^*} \Big[ \psi_{v_2^*}^n(t , Y_{v_2^*}(\tau))I \{ t > \tau\} \Big] \, ,}\\
\end{array}
\]
where
\[
\tau =\inf \{t\geq 0:Y_{v_2^*}(t)\in \{1,2,\cdots ,n\}\} .
\]

 In the last inequality we used the fact that
$\psi_{v_2^*}^n(\cdot, i)$ is nondecreasing in $t$ for each fixed $i$.
Hence
\[
\bar{\psi}_{v_2^*}^n(t, i) \ \leq \ 1+ \max_{j =1, \dots , n }
\bar{\psi}_{v_2^*}^n(t,j) \,
 \leq \, 1+  \max_{j =1, \dots , n }  W(j)  \, 
\]
since  $\psi_{v_2^*}^n(t, i_0) \geq 1$ and last inequality follows from  Lemma \ref{bd11}. Therefore for each $n \geq 1, \bar{\psi}_{v_2^*}^n$ is  bounded.
\end{proof}
\begin{remark}\label{rem2}
From Lemma \ref{bd12}, it follows that for each $n$, $\bar{\psi}_{v_2^*}^n$ is uniformly  bounded (in t and $i$) and that bound is independent of the $v_2^*$. Therefore, we conclude that for each $n$, $\bar{\psi}^n$ is also uniformly  bounded (in t and $i$), since $\bar{\psi}_{v_2^*}^n=\bar{\psi}^n$.
\end{remark}
\begin{lemma}\label{bd2}
Assume (A1) and $\|r\|_{\infty} \leq \delta$, where $\delta >0$ is given in (A1).
Then $$\displaystyle{\sup_{t \geq 0}}
 \Big \| \dfrac{1}{\psi^n(t,i_0)}\dfrac{d\psi^n}{dt}(t,\cdot) \Big \|_{W} < \infty .   $$
\end{lemma}
\begin{proof}
Note that
\begin{eqnarray*}
 \dfrac{1}{\psi^n(t,i_0)}\dfrac{d\psi^n}{dt}(t,i)&=& \dfrac{d\bar{\psi^n}}{dt}(t,i)+\dfrac{\bar{\psi^n}(t,i)}{\psi^n(t,i_0)}\dfrac{d\psi^n}{dt}(t,i_0) \nonumber \\
 &=& \inf_{v_1\in V_1} \sup_{v_2\in V_2}\Big [ \Pi_{v_1,v_2}
\bar\psi^n(t,i) + r_n(i,v_1,v_2)\bar\psi^n(t,i) \Big ]  \nonumber \\
&=& \sup_{v_2\in V_2} \inf_{v_1\in V_1} \Big [ \Pi_{v_1,v_2}
\bar\psi^n(t,i) + r_n(i,v_1,v_2)\bar\psi^n(t,i) \Big ] .
\end{eqnarray*}
The second equality follows from (\ref{fh_hjb7}).
Now the result follows from the fact $$\displaystyle{\sup_{i \in S, u \in U}} [-\bar{\pi}_{ii}(u)]=M<\infty $$ and Remark \ref{rem2}.
\end{proof}
Now we prove the existence of a solution to the HJI equation for the cost function $r_n$.
\begin{theorem}\label{bex-n}
 Assume (A1) and $\|r\|_{\infty} \leq \delta$, where $\delta >0$ is given in (A1). Then the  equation
\begin{eqnarray}\label{brs-rn}
 \rho^n ~\hat{\psi}^n(i) &=& \inf_{v_1\in V_1} \sup_{v_2\in V_2}\Big [ \Pi_{v_1,v_2}
\hat\psi^n(i) + r_n(i,v_1,v_2)\hat\psi^n(i) \Big ]  \nonumber \\
&=& \sup_{v_2\in V_2} \inf_{v_1\in V_1} \Big [ \Pi_{v_1,v_2}
\hat\psi^n(i) + r_n(i,v_1,v_2)\hat\psi^n(i) \Big ]
\end{eqnarray}
  has a solution $(\rho^n,\hat{\psi}^n(i))$ satisfying $\hat\psi^n(i_0) = 1$.
  Also
\begin{equation}
\rho^n \leq \sup_{v_2 \in \mathcal{M}_2} \inf_{v_1 \in \mathcal{M}_1}\limsup_{T\rightarrow \infty}
 \frac{1}{T} \ln  E_{i}^{v_1,v_2} \left[ e^{  \int_{0}^{T} r(Y(s),v_1(s,Y(s-)),v_2(s,Y(s-))) ds }  \right]
\end{equation}
and
\begin{equation}
0 < \hat{\psi}^n (i) \leq  W(i) , \ n \geq 1, i \in S.
\end{equation}
\end{theorem}
\begin{proof} Using mean value theorem and Remark \ref{rem2}, there exists $s(t, i) \in [t, \ 2t], \ t > 0$ such that
\[
\lim_{t \to \infty} \frac{d \bar{\psi}^n}{dt} (s(t, i), i) \ = \ 0 \, .
\]
 By Lemma \ref{bd11}, we have
$$\sup_{t \geq 0} |\bar{\psi}^n(s(t, i),i)| \leq  W(i) .$$
Using a diagonalization argument, along a subsequence,
$\bar{\psi}^n(s(t,i), i) \rightarrow \hat{\psi}^n( i)$, for each $i \in S$ for some $\hat{\psi}^n \in B_W(S)$.\\
By Lemma \ref{bd2}, we have
$$\displaystyle{\sup_{t \geq 0}} |\dfrac{1}{\psi^n(s(t,i),i_0)}\dfrac{d\psi^n}{dt}(s(t,i),i_0)| < \infty .$$
Therefore,  along a further subsequence denoted by the same notation by an abuse of notation, we have
$$\dfrac{1}{\psi^n(s(t,i),i_0)}\dfrac{d\psi^n}{dt}(s(t,i),i_0)\rightarrow \rho^n, \ {\rm for\ some} \ \rho^n \in
\mathbb{R}.$$ Hence, along a suitable subsequence, we have 
$$\dfrac{d\bar{\psi^n}}{dt}(s(t,i),i)+\dfrac{\bar{\psi^n}(s(t,i),i)}{\psi^n(s(t,i),i_0)}\dfrac{d\psi^n}{dt}(s(t,i),i_0)
\rightarrow \hat{\psi}^n(i)\rho^n.$$
By letting $t \rightarrow \infty $ in (\ref{fh_hjb7}) at $t = s(t, i)$ along a suitable subsequence, and using (A1)
it follows  that $(\rho^n,\hat{\psi}^n(i))$ is a solution to the equation (\ref{brs-rn}) satisfying
$\hat\psi^n(i_0) = 1$.

Let $v^*_n=(v_{1n}^*,v_{2n}^*):S \to V$ be a min-max selector such that
\begin{eqnarray}\label{eqbdd}
 && \sup_{v_2\in V_2}\Big [ \Pi_{v_{1n}^*(i),v_2}
\hat\psi^n(i) + r_n(i,v_{1n}^*(i),v_2)\hat\psi^n(i) \Big ] \nonumber \\
&=& \inf_{v_1\in V_1} \sup_{v_2\in V_2}  \Big [ \Pi_{v_1,v_2}
\hat\psi^n(i) + r_n(i,v_1,v_2)\hat\psi^n(i) \Big ] \nonumber \\
&=& \sup_{v_2\in V_2} \inf_{v_1\in V_1} \Big [ \Pi_{v_1,v_2}
\hat\psi^n(i) + r_n(i,v_1,v_2)\hat\psi^n(i) \Big ] \nonumber \\
&=&  \inf_{v_1\in V_1} \Big [ \Pi_{v_1,v_{2n}^*(i)}
\hat\psi^n(i) + r_n(i,v_1,v_{2n}^*(i))\hat\psi^n(i) \Big ]
\end{eqnarray}
For $v_n:=(v_1,v_{2n}^*) \in {\mathcal{M}_1}\times \mathcal{M}_2$, let $Y(\cdot)$ be the
process (\ref{cmc}) corresponding to $v_n$ with
initial condition $i \in S$.
Then using It$\hat{\rm o}$-Dynkin's formula and (\ref{eqbdd}), we get
\begin{eqnarray*}
 E_{i}^{v_1,v_{2n}^*} \Big [ e^{  \int_{0}^{T } (r_n(Y(s),v_1(s, Y(s-)),v_{2n}^*(s, Y(s-)))-\rho_n) ds } \hat{\psi}^n(Y(T))\Big ]
 -\hat{\psi}^n(i)&\geq& 0  .
\end{eqnarray*}
From Remark \ref{rem2} it follows that for each $n$, $\hat{\psi}^n$ is bounded. Therefore we have
\begin{eqnarray*}
 \hat{\psi}^n(i)  \leq K(n) E_{i}^{v_1,v_{2n}^*} \Big [ e^{  \int_{0}^{T } (r_n(Y(s),v_1(s, Y(s-)),
v_{2n}^*(s, Y(s-)))-\rho_n) ds }  \Big ].
\end{eqnarray*}
Taking logarithm, dividing by $T$ and by letting $T \to \infty$, we get
\begin{eqnarray*}
 \rho_n \leq \limsup_{T\rightarrow \infty}\dfrac{1}{T}\ln E_{i}^{v_1,v_{2n}^*} \Big [ e^{  \int_{0}^{T} r_n(Y(s),v_1(s, Y(s-)),v_{2n}^*(s, Y(s-))) ds } \Big ].
\end{eqnarray*}
Since $v_1 \in {\mathcal M_1}$ is arbitrary, it follows that
\begin{equation*}
\rho_n \leq \inf_{v_1 \in \mathcal M_1}\limsup_{T\rightarrow \infty} \frac{1}{T} \ln  E_{i}^{v_1,v_{2n}^*}
 \left[ e^{ \int_{0}^{T} r_n(Y(s),v_1(s, Y(s-)),v_{2n}^*(s, Y(s-))) ds }  \right] .
\end{equation*}
Therefore we have
\begin{equation*}
\rho_n \leq \sup_{v_2 \in \mathcal M_2} \inf_{v_1 \in \mathcal M_1}\limsup_{T\rightarrow \infty} \frac{1}{T} \ln  E_{i}^{v_1,v_2}
\left[ e^{  \int_{0}^{T} r_n(Y(s),v_1(s, Y(s-)),v_2(s, Y(s-))) ds }  \right] .
\end{equation*}
Since $r_n \leq r $, we have
\begin{equation*}
\rho_n \leq \sup_{v_2 \in \mathcal M_2} \inf_{v_1 \in \mathcal M_1}\limsup_{T\rightarrow \infty}
 \frac{1}{T} \ln  E_{i}^{v_1,v_2}
\left[ e^{  \int_{0}^{T} r(Y(s),v_1(s,Y(s-)),v_2(s,Y(s-))) ds }  \right].
\end{equation*}
\end{proof}
 Now by using Theorem \ref{bex-n}, one can closely mimic the arguments in the proof of [\cite{SP1}, Theorem 3.3] to prove the following.
\begin{theorem}\label{bex}
Assume (A1) and $\|r\|_{\infty} \leq \delta$, where $\delta >0$ is given in (A1). Then the  equation
(\ref{brs}) has a solution $(\rho,\hat{\psi}(i))$ satisfying $\hat\psi(i_0) = 1$.
  Also
\begin{equation}
\rho \leq \sup_{v_2 \in \mathcal{M}_2} \inf_{v_1 \in \mathcal{M}_1}\limsup_{T\rightarrow \infty}
 \frac{1}{T} \ln  E_{i}^{v_1,v_2} \left[ e^{  \int_{0}^{T} r(Y(s),v_1(s,Y(s-)),v_2(s,Y(s-))) ds }  \right] .
\end{equation}
\end{theorem}
To prove that the $\rho$ in Theorem \ref{bex} is indeed the value of the  game.  We used the atomic structure of the state dynamics, as in \cite{SP1}.
Let $v^*_1 \in \mathcal{S}_1$ be the outer minimizing selector in (\ref{brs}) for player 1, $v_2 \in \mathcal{M}_2$  any strategy for player 2 and let $Y(\cdot)$ be a continuous time Markov chain corresponding to $(v^*_1,v_2) \in \mathcal{S}_1 \times \mathcal{M}_2 $.

Define the  twisted kernel $\tilde{P}(j,i)$ associated with $Y(\cdot)$ and $r$ as follows.
\begin{equation}\label{twisted}
\sum_{j \in S} h(j)\tilde{P}(j,i)\, = \,
\dfrac{E_{i}^{v_1^*,v_2}[e^{\int_0^1 r(Y(s),v^*_1(Y(s-)),v_2(s,Y(s-)))ds}h(Y(1))]}{E_{i}^{v_1^*,v_2}[e^{\int_0^1 r(Y(s),v^*_1(Y(s-)),v_2(s,Y(s-)))ds}]}, i \in S , \
h \in B(S) .
\end{equation}
Set
$$ e^{\hat r(i)}=E_{i}^{v_1^*,v_2}[e^{\int_0^1 r(Y(s),v^*_1(Y(s-)),v_2(s,Y(s-)))ds}].$$
Let $\{ \tilde{Y}_n\}$ be a  Markov chain on some
probability space $(\tilde{\Omega},\tilde{\mathcal{F}},\tilde{P} )$
 with transition kernel $\tilde{P}(\cdot, \cdot)$ and initial condition $i \in S$.
We will denote the corresponding expectation by $\tilde{E}_i[\cdot]$.\\
 Fix $i\in S$. Let $\{\tilde{Y}_n\}$ be the Markov chain given by (\ref{twisted}) with $\tilde{Y}_0=i$. Set
 $\tilde{\tau}=\inf\{n\geq 1| \tilde{Y}_n=i \}:=\tilde{\tau}_1$ and
 $\tilde{\tau}_{n+1}=\inf\{n\geq \tilde{\tau}_n+1| \tilde{Y}_n=i \}$. Define
$$ D(\rho)= \tilde E_{i}[e^{\sum_{n=1}^{\tilde \tau}(\hat r(\tilde Y_n)-\rho)}].$$
We state the following lemmas which play a crucial role in this section. Since the proofs of these results closely mimic the corresponding proofs in \cite{SP1}, we omit the details.
\begin{lemma} Assume (A1) and $\|r\|_{\infty} < \delta$. Then
\begin{eqnarray}\label{D1}
D(\rho) \leq 1.
\end{eqnarray}
\end{lemma}
\begin{lemma}\label{tkep}
Assume (A1). Then for each $i \in S$ such that
$W(i) \geq  1 + \frac{b e^{\frac{3\delta}{2} }}{e^{\frac{\delta}{2}} -1}$, 
 $$\tilde{E}_i[e^{\delta \tilde{\tau}/2}] \leq  e^{-\delta/2}(W(i)+b e^{3 \delta/2}),$$  where $\delta >0$ is given in (A1).
\end{lemma}
Define $C_0=\{i\in S: W(i) \geq  1 + \frac{b e^{\frac{3\delta}{2} }}{e^{\frac{\delta}{2}} -1}\}$.
Now we state and prove the main theorem.
\begin{theorem}\label{ch}
 Assume (A1) and  $\|r\|_{\infty} < \frac{\delta}{2} $, where $\delta >0$ is given in (A1). Let $(\rho,\hat{\psi}(i))$ be the solution obtained in Theorem \ref{bex}.
 Then
\begin{eqnarray*}
\rho &=& \inf_{v_1(\cdot) \in \mathcal{M}_1}\sup_{v_2(\cdot)\in \mathcal{M}_2}
\limsup_{T\rightarrow \infty} \frac{1}{T}
\ln  E_{i}^{v_1,v_2} \left[ e^{  \int_{0}^{T} r(Y(s),v_1(s,Y(s-)),v_2(s,Y(s-))) ds }  \right]  \\
&=&\sup_{v_2(\cdot)\in \mathcal{M}_2} \inf_{v_1(\cdot) \in \mathcal{M}_1}\limsup_{T\rightarrow \infty} \frac{1}{T}
\ln  E_{i}^{v_1,v_2} \left[ e^{  \int_{0}^{T} r(Y(s),v_1(s,Y(s-)),v_2(s,Y(s-))) ds }  \right] ,
\end{eqnarray*}
 i.e., $\rho $ is the value of the risk-sensitive ergodic game. Furthermore there exists a pair of saddle point
stationary Markov strategies $(v_1^*,v_2^*)$ such that  $v_1^*$ is the outer minimizing selector in (\ref{brs}), and $v_2^*$ is the outer maximizing selector in (\ref{brs}).
\end{theorem}
\begin{proof}
 In view of Theorem \ref{bex}, it remains to show that
\begin{equation*}
\rho \geq \inf_{v_1(\cdot) \in \mathcal{M}_1}\sup_{v_2(\cdot)\in \mathcal{M}_2}
\limsup_{T\rightarrow \infty} \frac{1}{T}
\ln  E_{i}^{v_1,v_2} \left[ e^{  \int_{0}^{T} r(Y(s),v_1(s,Y(s-)),v_2(s,Y(s-))) ds }  \right],
\end{equation*}
and the existence of  a pair of saddle point
stationary Markov strategies $(v_1^*,v_2^*)$ such that  $v_1^*$ is the outer minimizing selector in (\ref{brs}), and $v_2^*$ is the outer maximizing selector in (\ref{brs}).

Fix an $i \in C_0$ and  $N \in \mathbb{N}$,
define $$e^{B_N(i)}=\tilde E_{i}
 \left[ \exp   \{   \sum_{k=1}^{ N \wedge \tilde {\tau} }(\hat r(\tilde Y_k)-\rho) \}  \right] \; \mbox{for} \; N \in \mathbb{N} .$$
 Arguing as in the proof of [\cite{SP1}, Theorem 3.5], it follows that 
\begin{eqnarray*}
e^{B_{N+k}(i)}
&\geq& \tilde E_{i} \Big [\exp   \{ \sum_{m=1}^k  (\hat r(\tilde Y_m)-\rho)-N(\|\hat{r}\|_{\infty}+\rho)\} \Big ].
\end{eqnarray*}
From Lemma \ref{tkep}, it follows that for each $i\in C_0$, $ e^{B_N(i)}\leq e^{-\delta/2}(W(i)+b e^{3 \delta/2})  $. Therefore
taking logarithm in both sides and letting $k \rightarrow \infty$ we get
$$  \limsup_k \frac{1}{k} \ln \tilde E_{i} \Big [\exp   \{ \sum_{m=0}^{k-1}  \hat r(\tilde Y_m)\} \Big ]\leq \rho .$$
By using mathematical induction, we show that $\forall k \in \mathbb{N}$
$$ E_{i}^{v_1^*,v_2} \Big [\exp   \{ \int_{0}^{k}   r( Y(s),v_1^{*}(Y(s-)),v_2(s,Y(s-)))ds\} \Big ]=
 \tilde E_{i} \Big [\exp   \{ \sum_{i=0}^{k-1}  \hat r(\tilde Y_i)\} \Big ].$$
That is true for $k=1$, by definition. Let this be true for $k=n$. Then for $k=n+1$
\begin{eqnarray*}
&& \tilde E_{i} \Big [\exp   \{ \sum_{i=0}^{n}  \hat r(\tilde Y_i)\} \Big ] \\
&=&
\tilde E_{i} \Big [\exp \{ \hat r(i) \}\tilde E_{\tilde Y_{1}}
[\exp   \{ \sum_{i=0}^{n-1}  \hat r(\tilde Y_i)\}]\Big ]\\ \nonumber
&=& \tilde E_{i} \Big [\exp \{ \hat r(i) \} E_{\tilde Y_{1}}
\Big [\exp   \{ \int_{0}^{n}   r( Y(s),v_1^{*}(Y(s-)),v_2(s,Y(s-)))ds\} \Big ]\Big ] \\ \nonumber
&=& E_{i}^{v_1^*,v_2}\Big [ \exp   \{ \int_{0}^{1}   r( Y(s),v_1^{*}(Y(s-)),v_2(s,Y(s-)))ds\} \\
&&E_{ Y(1)}^{v_1^*,v_2} \Big [\exp   \{ \int_{0}^{n}  r( Y(s),v_1^{*}(Y(s-)),v_2(s,Y(s-))))ds\} \Big ] \Big ] \\ \nonumber
&=& E_{i}^{v_1^*,v_2}\Big [ \exp   \{ \int_{0}^{n+1}   r( Y(s),v_1^{*}(Y(s-)),v_2(s,Y(s-)))ds\}  \Big ] .
\end{eqnarray*}
Hence we get
\begin{equation*}
\rho \geq \limsup_{T\rightarrow \infty} \frac{1}{T}
\ln  E_{i}^{v_1^*,v_2} \left[ e^{  \int_{0}^{T} r( Y(s),v_1^{*}(Y(s-)),v_2(s,Y(s-))) ds }  \right], i \in C_0.
\end{equation*}
Since $v_2 \in \mathcal{M}_2$ is arbitrary, it follows that
\begin{equation*}
\rho \geq \sup_{v_2 \in \mathcal M_2}\limsup_{T\rightarrow \infty} \frac{1}{T}
\ln  E_{i}^{v_1^*,v_2} \left[ e^{  \int_{0}^{T} r( Y(s),v_1^{*}(Y(s-)),v_2(s,Y(s-))) ds }  \right], i \in C_0.
\end{equation*}
Therefore we have
\begin{equation*}
\rho \geq \inf_{v_1 \in \mathcal M_1} \sup_{v_2 \in \mathcal M_2} \limsup_{T\rightarrow \infty} \frac{1}{T}
\ln  E_{i}^{v_1^*,v_2} \left[ e^{  \int_{0}^{T} r( Y(s),v_1^{*}(Y(s-)),v_2(s,Y(s-))) ds }  \right], i \in C_0.
\end{equation*}
Thus
\begin{equation*}
\rho = \limsup_{T\rightarrow \infty} \frac{1}{T}
\ln E_{i}^{v_1^*,v_2^*} \left[ e^{  \int_{0}^{T} r( Y(s),v_1^{*}(Y(s-)),v_2^*(Y(s-))) ds }  \right], i \in C_0.
\end{equation*} 
Arguing as in the proof of [\cite{SP1}, Theorem 3.5], it follows that the above equation holds for all $i \in S$, i.e.,
 \begin{equation*}
\rho = \limsup_{T\rightarrow \infty} \frac{1}{T}
\ln  E_{i}^{v_1^*,v_2^*} \left[ e^{  \int_{0}^{T} r( Y(s),v_1^{*}(Y(s-)),v_2^*(Y(s-))) ds }  \right], i \in S.
\end{equation*}
It is easy to check that $\rho $ is the value of the risk-sensitive ergodic game. Moreover, the pair of stationary Markov strategies $(v_1^*,v_2^*)$ where  $v_1^*$ is the outer minimizing selector in (\ref{brs}), $v_2^*$ is the outer maximizing selector in (\ref{brs}) forms a saddle point equilibrium.
This completes the proof.
\end{proof}

\section{Conclusions}
We have studied zero-sum risk-sensitive stochastic games for
continuous time Markov chains on a countable state space. For the ergodic case we have taken the risk sensitive parameter $\theta=1$ for the sake of simplicity. If we choose any other $\theta\in (0,\Theta)$, then the assumption in Theorem \ref{ch} has to be modified to $\theta\|r\|_{\infty}<\dfrac{\delta}{2}$. This is the so called small cost criterion which is standard in the literature \cite{BG1},
\cite{BG2}, \cite{GhoshSaha}, \cite{SP1}.  The
corresponding non-zero sum case is currently under investigation.


~\\{\bf Acknowledgments: } \\
The work of the first named author is supported in part by UGC Centre for Advanced Study. The work of the second named author is supported in part by the DST, India 
project no. SR/S4/MS:751/12. The work of the third named author is supported in part by Dr. D. S. Kothari postdoctoral fellowship of UGC.




\bibliographystyle{elsarticle-num} 

\end{document}